\documentclass[11pt,leqno]{amsart}
\usepackage[a4paper,top=2.5 cm,bottom=2 cm,left=2.5 cm,right=2 cm]{geometry}
\usepackage{graphicx}
\usepackage{upref,ulem}
\usepackage{textcomp} 
\usepackage{tikz} 
\usetikzlibrary{arrows,snakes,backgrounds,calc}
\usetikzlibrary{hobby,decorations.markings}
\usepackage{amsmath ,amssymb}
\usepackage{color,bbm,epsfig,float,fancyhdr,soul}
\usepackage[mathscr]{eucal}
\usepackage{caption}
\usepackage[makeroom]{cancel}
\usepackage{bm}
\usepackage{multirow}
\usepackage{colortbl}
\DeclareMathOperator{\sgn}{sgn}

\newtheorem{definition}{Definition}[section]
\newtheorem{theorem}{Theorem}[section]

\newtheorem*{maintheorem*}{Main Theorem}

\allowdisplaybreaks
\numberwithin{equation}{section}
\newtheorem{assumptions}{Assumptions}[section]

\renewcommand{\i}{\ifmmode\mathit{\mathchar"7010 }\else\char"10 \fi}
\renewcommand{\j}{\ifmmode\mathit{\mathchar"7011 }\else\char"11 \fi}
\newcommand{\R}{\mathbb{R}}

\newcommand{\Cc}[1]{\mathbf{C_c^{#1}}}

\newcommand{\dd}{\mathrm{d}}

\newcommand{\BV}{\mathbf{BV}}
\renewcommand{\L}[1]{\mathbf{L^#1}}

\newcommand{\C}[1]{\mathbf{C^{#1}}}
{%

\begin{enumerate}}%
{\end{enumerate}}

{%

\begin{enumerate}}%
{\end{enumerate}}

\title{Stability estimates for nonlocal balance laws arising in traffic modelling}
\author{Felisia Angela Chiarello } 
\address[Felisia Angela Chiarello]{\newline 
Dipartimento di Ingegneria e Scienze dell'Informazione e Matematica
\newline Via Vetoio, Ed. Coppito 1 67100 L'Aquila AQ, Italy}
\email[]{felisiaangela.chiarello@univaq.it}
\author{Harold Deivi Contreras}
\address[Harold Deivi Contreras]{\newline
GIMNAP-Departamento de Matem\'aticas, Universidad del B\'io-B\'io, Concepci\'on, Chile,\newline CI${}^2$MA-Universidad de Concepci\'on, Casilla 160-C, Concepci\'on, Chile.}
\email[]{harold.contreras1801@alumnos.ubiobio.cl}

\begin{document}

\begin{abstract}
This paper focuses on the proof of the stability of entropy weak solutions of a nonlocal balance law modelling vehicular traffic flow on a road with on- and off-ramps. The stability is obtained  with respect to a kernel function in the source term. We get an estimate of the $\L1-$dependence of the solution with respect to the initial datum, the on-ramp rate, the off-ramp rate and the mentioned kernel function. We also show numerical experiments in order to perform an optimization problem in traffic flow with on-ramps.

\end{abstract}

\date{\today}

\maketitle

\section{Introduction}
\subsection{Aim}
In traffic flow modeling, nonlocal conservation laws are intended to describe the behaviour of drivers that adapt their velocity with respect to what happens in front of them, see \cite{Blandin2016WellposednessOA, chiarello2020non, chiarello2019non, friedrich2021nonlocal, FKG2018};  the classical LWR (Lighthill-Whitham \cite{LW} and Richards \cite{R}) is not able to model this type of situation. For this reason, in order to generalise the LWR model in such a way to make it able to describe more realistic settings, in \cite{chiarello2022nonlocal} a nonlocal balance law intended to model vehicular traffic flow on a main road wit on- and off-ramps is introduced  and it is given by
\begin{equation}\label{nonRTM_2}
    \rho_t+(\rho v(\rho*\omega_{\eta}))_x=S_{\mathrm{on}}(\cdot, \cdot, \rho,\rho*\omega_{\eta,\delta})-S_{\mathrm{off}}(\cdot, \cdot, \rho), \quad x\in\R,
\end{equation}
where $S_{\mathrm{on}}$ and $S_{\mathrm{off}}$ represents the  traffic flow entering and exiting through an on- and off-ramp respectively, and  the convolution term in $S_{\mathrm{on}}$ is defined as follows
$$(\rho*\omega_{\eta,\delta})(t,x)=\int_{x-\eta+\delta}^{x+\eta+\delta}\rho(t,y)\omega_{\eta,\delta}(y-x)\dd y,$$ 
with $\eta\in[0,1]$ and $\delta\in[-\eta,\eta].$ Here, the parameter $\eta$ represents the length of the support of the kernel function $\omega_{\eta,\delta}$, while $\delta$ is the point at which the maximum of the kernel is attained. 
This choice of the kernel function models the fact that drivers on the on-ramp can see what happens on backward and forward on the main road.\\
It is well known that on-ramp merging has a great impact on traffic efficiency, if one concentrates on highway
networks the reduction of the capacity is often due to on- and off-ramps, for this reason in this paper we are interested to study the dependence of solutions to \eqref{nonRTM_2} on the convolution kernel given in the source term $S_{\mathrm{on}}$. The strategies that we employ are inspired by the results obtained in \cite{chiarello2022nonlocal, chiarello2019stability}. In particular, we adopt the results about existence and uniqueness to~\eqref{nonRTM_2} presented in~\cite{chiarello2022nonlocal} and we propose the analytical study of the dependence of solutions to~\eqref{nonRTM_2} on the kernel function in the source term.
\subsection{Related work}
In literature, the LWR model has been extended to include on and off-ramps by means of source and sink terms, see  \cite{delle2014pde, helbing2001master, liptak2021traffic,  liu1996modelling, tie2010effects, tie2009new}. Starting from the model introduced in \cite{liptak2021traffic}, in \cite{chiarello2022nonlocal} the authors introduce the new nonlocal balance law~\eqref{nonRTM_2}, in which the velocity function depends on a weighted mean of the downstream traffic density, in order to model the behaviour of drivers that adapt their velocity with respect to what happens in front of them; moreover the nonlocal source term is intended to model the fact that drivers on the on-ramp can see what happens on backward and forward on the main road. 
Regarding the aferomentioned model the existence and uniqueness of entropy weak solution has been proved, specifically,  approximating the problem using an Upwind-type numerical scheme and providing compactness estimates for the sequence of approximate solutions that allows to get the convergence and thus, the existence of solution. Furthemore, by using Kru\v{z}kov's doubling of variables technique the authors get an $\L{1}$ stability estimate depending on initial data, the on-ramp rate and the off-ramp rate. 
On the other hand, motivated by control and optimisation problems, in \cite{chiarello2019stability} the authors study a conservation law with nonlocal flux function given by
\begin{eqnarray}\label{eq:stability_Chiarello}
\partial_{t}\rho+\partial_{x}\left(f(t,x,\rho)v(\rho\ast\omega)(t,x)\right)=0.
\end{eqnarray}
The authors obtain an estimate of the dependence of the solution to~\eqref{eq:stability_Chiarello} with respect to the kernel function, the velocity and the initial datum. The stability is obtained from an entropy condition through the doubling of variables technique. 
\subsection{Outline of this paper.} This work is organized as follows: In Section \ref{sec:model} we recall the mathematical model on which we will focus our study as well as the definitions of weak and entropy weak solution. In section~\ref{sec:Main} we present the main results of this work, i.e. the $\L{1}-$ Lipschitz continuous dependence of the solutions to \eqref{nonRTM_2} on
to the initial datum, the on-ramp rate, the off-ramp rate and the kernel function in the source term. Finally, in Section~\ref{sec:Num_exp} we present numerical examples in order to illustrate the behaviour of solutions when some parameters of the kernel function in the source term are varying. 
\section{Mathematical model }\label{sec:model}
Let us consider the equation~\eqref{nonRTM_2} with terms $S_{\mathrm{on}}$ and $S_{\mathrm{off}}$ defined as     
\begin{eqnarray}
\label{Son_Hyp} S_{\mathrm{on}}(t, x, \rho,\rho*\omega_{\eta,\delta}) &=&\mathbf{1}_{\mathrm{on}}(x) q_{\mathrm{on}}(t) \left(1-\frac{\rho}{\rho_{\max}}\right)\left(1-\frac{\rho*\omega_{\eta,\delta}}{\rho_{\max}}\right),\\
\label{Soff_Hyp}S_{\mathrm{off}}(t, x, \rho) &=&\mathbf{1}_{\mathrm{off}}(x)q_{\mathrm{off}}(t)\frac{\rho}{\rho_{\max}},
\end{eqnarray}
with $\rho_{\max}=1$ for the sake of simplicity, and we also endow the nonlocal traffic reaction model~\eqref{nonRTM_2}-\eqref{Son_Hyp}-\eqref{Soff_Hyp} with an initial condition as follows  
\begin{eqnarray}\label{initial_condition}
 \rho(x,0)=\rho_0(x)\in \left(\L{1}\cap\BV\right)(\R;[0,{\rho_{\max}}]).
\end{eqnarray}
Let us assume the following assumptions on the parameters of model~\eqref{nonRTM_2}.\\
\begin{assumptions}\label{Ass}
We assume
\begin{itemize}
\item[($i$)] $q_{\mathrm{on}}^{\mathrm{ramp}}\in\L{\infty}(\R^{+};\R^{+}),q_{\mathrm{off}}^{\mathrm{ramp}}\in\L{\infty}(\R^{+};\R^{+}).$
\item[($ii$)] $v\in\C{2}([0, {\rho_{\max}}];\R^{+}), \ v'(\rho)\leq0,\, \rho\in[0,\rho_{\max}]$.
 \item[($iii$)] $\omega_{\eta}\in\C{1}([0,\eta];\R^{+})\ \textup{with} \  \omega'_{\eta}(x)\leq0, \ \int_{0}^{\eta}\omega_{\eta}(x)\dd x=1, \ \forall\eta>0$.
\item[($iv$)] $\omega_{\eta,\delta}\in(\C{1}\cap\L{1})([\delta-\eta,\delta+\eta];\R^{+}) \ \textup{with} \ \omega'(x)_{\eta,\delta}\geq0 \ \text{for} \ x\in[\delta-\eta,0]$,
$\omega'(x)_{\eta,\delta}\leq0 \ \text{ for } x\in[0,\delta+\eta]$, {and} $\int_{\delta-\eta}^{\delta+\eta}\omega_{\eta,\delta}(x)\dd x=1,\ \forall\eta>0.$
\end{itemize}
\end{assumptions}
We consider solutions in a weak sense as follows 
\begin{definition}\label{weak_Sol}
Let $\rho_{0}\in (\L{1}\cap\BV)(\R;[0, {\rho_{\max}}]).$ We say that $\rho\in\mathbf{C}([0,T]; \L{1}(\R;[0, {\rho_{\max}}])),$ with $\rho(t,\cdot)\in\BV(\R;[0, {\rho_{\max}}])$ for $t\in[0,T]$, is a weak solution to~\eqref{nonRTM_2}-\eqref{initial_condition} if for any $\varphi\in\Cc{1}([0,T[\times\R;\R) $
\begin{eqnarray*}
\int_{0}^{T}\int_{\R}\left(\rho\varphi_{t}+\rho V\varphi_{x}\right)\dd x\dd t+\int_{0}^{T}\int_{\Omega_{\mathrm{on}}}S_{\mathrm{on}}\varphi\dd x\dd t\\
-\int_{0}^{T}\int_{\Omega_{\mathrm{off}}}S_{\mathrm{off}}\varphi\dd x\dd t+\int_{\R}\rho_{0}(x)\varphi(0,x)\dd x=0,
\end{eqnarray*}
where $V(t,x)=v((\rho*\omega)(t,x))$. 
\end{definition}
Furthermore, we consider entropy weak solutions in  Kru\v zkov sense, i.e.
\begin{definition}\label{entropy_WS}
Let $\rho_{0}\in (\L{1}\cap\BV)(\R;[0, {\rho_{\max}}]).$ We say that $\rho\in\mathbf{C}([0,T]; \L{1}(\R;[0, {\rho_{\max}}])),$ with $\rho(t,\cdot)\in\BV(\R;[0, {\rho_{\max}}])$ for $t\in[0,T]$, is a entropy weak solution to \eqref{nonRTM_2} with initial datum $\rho_{0}$ if for any $\varphi\in\Cc{1}([0,T[\times\R;\R)$ and for all $k\in\R$
\begin{eqnarray*}
\int_{0}^{T}\int_{\R}\left(\vert\rho-k \vert\varphi_{t}+\left|\rho-k\right| V\varphi_{x}-\sgn(\rho-k)k V_{x }\varphi\right)\dd x\dd t\\
+\int_{0}^{T}\int_{\Omega_{\mathrm{on}}}\sgn(\rho-k)S_{\mathrm{on}}\varphi\dd x \dd t-\int_{0}^{T}\int_{\Omega_{\mathrm{off}}}\sgn(\rho-k)S_{\mathrm{off}}\varphi\dd x\dd t\\
+\int_{\R}\vert\rho_{0}-k \vert\varphi(0,x)\dd x\geq0,
\end{eqnarray*} 
\end{definition}
where $\Omega_{\mathrm{on}}$ and $\Omega_{\mathrm{off}}$ are the spatial positions of on-ramp and off-ramp on the main road, respectively. 

\section{Main Result}\label{sec:Main}
Before giving the main results of this work, we first recall the main Theorem in \cite[Theorem 2.1]{chiarello2022nonlocal}.  
\begin{theorem}
Let $\rho_{0}\in\left(\L{1}\cap\BV\right)\left(\R;[0,1]\right)$. Let the Assumptions \ref{Ass} hold. Then, for all $T>0$, the problem \eqref{nonRTM_2} has a unique solution $\rho\in\C{0}\left([0,T];\L{1}(\R;[0,1])\right)$ in the sense of Definition \ref{entropy_WS}. Moreover, the following estimates hold: for any $t\in[0,T]$
\begin{equation*}
\begin{array}{ll}
\left\|\rho(t)\right\|_{\L{1}(\R)}\leq \mathcal{R}_{1}(t),\\
0\leq\rho(t,x)\leq1,\\
TV(\rho(t))\leq e^{t\mathcal{H}}\left(TV(\rho_{0})+t\mathcal{Q}_T \right),
\end{array}
\end{equation*}
where 
\begin{equation}\label{R_1}\
\begin{split}
\mathcal{R}_{1}&=\left\|\rho_{0}\right\|_{\L{1}(\R)}+\left\|q_{\mathrm{on}}^{\mathrm{ramp}}(\cdot)\right\|_{\L{1}([0,t])}-\min_{x\in\Omega_{\mathrm{on}}}\left\|q_{\mathrm{on}}^{\mathrm{ramp}}(\cdot)\rho(\cdot,x)\right\|_{\L{1}([0,t])}\\
&-\min_{x\in\Omega_{\mathrm{off}}}\left\|q_{\mathrm{off}}^{\mathrm{ramp}}(\cdot)\rho(\cdot,x)\right\|_{\L{1}([0,t])}\\
\end{split}
\end{equation}
\begin{eqnarray}
\label{eq:QT}\mathcal{Q}_T&=&{2}\left(\left\|q_{\mathrm{on}}\right\|_{\L{\infty}([0,T])}+\left\|q_{\mathrm{off}}\right\|_{\L{\infty}([0,T])}\right) \\
\label{H}\mathcal{H}&=&2\left\|q_{\mathrm{on}}\right\|_{\L{\infty}([0,T])}+\left\|q_{\mathrm{off}}\right\|_{\L{\infty}([0,T])}+\omega_{\eta}(0)\mathcal{L}\\
\label{L}\mathcal{L}&=&\left(\|v\|_{\L{\infty}([0,1])}+\|v'\|_{\L{\infty}([0,1])}\right).
\end{eqnarray}
\end{theorem}

The following theorem is the main result of this work and it states the $\L{1}$- Lipschitz continuous dependence of solutions to \eqref{nonRTM_2} on to the initial datum, the on-ramp rate, the off-ramp rate and the kernel function. 

\begin{theorem}\label{Lips_Uniq}
Let $\rho$ and $\tilde{\rho}$ be two solutions to problem \eqref{nonRTM_2} in the sense of Definition \ref{entropy_WS}, with initial data $\rho_0,\ \tilde{\rho}_{0}\in\L{1}\cap\BV\left(\R;[0,1]\right)$, with on-ramp rates $q_{\mathrm{on}}$, $\tilde{q}_{\mathrm{on}}$, off-ramp rates  $q_{\mathrm{off}}$, $\tilde{q}_{\mathrm{off}}$ and kernel functions ${\omega}_{\eta,\delta},\,$ $\tilde{\omega}_{\eta,\delta}$, respectively. Assume $v\in\C{2}\left([0,1],\R^{+}\right)$. Then, for a.e. $t\in[0,T]$,
\begin{eqnarray*}
\left\|\rho(t)-\tilde{\rho}(t)\right\|_{\L{1}(\R)}&\leq&e^{\mathcal{C}T}\bigg(\left\|\rho_{0}-\tilde{\rho}_{0}\right\|_{\L{1}(\R)}+ {L}\left(\left\|q_{\mathrm{on}}-\tilde{q}_{\mathrm{on}}\right\|_{\L{1}([0,t])}+\left\|q_{\mathrm{off}}-\tilde{q}_{\mathrm{off}}\right\|_{\L{1}([0,T])}\right)\\
&&+r(T)\|\omega_{\eta,\delta}-\tilde{\omega}_{\eta,\delta}\|_{\L{1}(\R)}\bigg),
\end{eqnarray*}
where $r(T)$ depends on the $\L{1}-$norms of initial conditions, expected inflow flow of the on-ramp and expected output flow of the off-ramp and $\mathcal{C}$ is defined as in~\cite[(3.24)]{chiarello2022nonlocal}. 
\end{theorem}
\begin{proof}
Since $\rho$ and $\tilde{\rho}$ are entropy weak solutions to~\eqref{nonRTM_2} then
\begin{eqnarray*}\label{problems}
\left\{ \begin{array}{lcc}
\rho_{t}+(\rho V(t,x))_{x}=S_{\mathrm{on}}\left(t,x,q_{\mathrm{on}},\rho,R_{\mathrm{on}}\right)-S_{\mathrm{off}}\left(t,x,q_{\mathrm{off}},\rho\right)\\
\rho(0,x)=\rho_{0}(x),
 \end{array}
\right.
   \end{eqnarray*}
   \begin{eqnarray*}
 \left\{ \begin{array}{lcc}
\tilde{\rho}_{t}+(\tilde{\rho} \tilde{V}(t,x))_{x}={S}_{\mathrm{on}}\left(t,x,\tilde{q}_{\mathrm{on}},\tilde{\rho},\tilde{R}_{\mathrm{on}}\right)-S_{\mathrm{off}}\left(t,x,\tilde{q}_{\mathrm{off}},\rho\right)\\
\tilde{\rho}(0,x)=\tilde{\rho}_{0}(x),
 \end{array}\right.
 \end{eqnarray*}
in a distributional sense, where\ $ V(t,x)=v((\rho\ast\omega_{\eta})(t,x)),$ $R_{\mathrm{on}}=(\rho\ast\omega_{\eta,\delta})(t,x)$, $ \tilde{V}(t,x)=v((\tilde{\rho}\ast\omega_{\eta})(t,x)),$ $\tilde{R}_{\mathrm{on}}=(\tilde{\rho}\ast\tilde{\omega}_{\eta,\delta})(t,x)$. Following the argument in~\cite{chiarello2022nonlocal} and using Kru\v{z}kov's doubling of variables technique we get 
\begin{eqnarray}\label{doubling_variables}
\left\|\rho(T,\cdot)-\tilde{\rho}(T,\cdot)\right\|_{\L{1}(\R)}&\leq&\left\|\rho_{0}-\tilde{\rho}_{0}\right\|_{\L{1}(\R)}+\int_{0}^{T}\int_{\Omega_{\mathrm{on}}}\left|\tilde{\mathcal{S}}_{\mathrm{on}}\right|\dd x\dd t+\int_{0}^{T}\int_{\Omega_{\mathrm{off}}}\left|\tilde{\mathcal{S}}_{\mathrm{off}}\right|\dd x\dd t\nonumber\\
&&+\int_{0}^{T}\int_{\R}\left|\mathcal{V}\right|\left|\partial_{x}\rho(t,x)\right|\dd x \dd t+\int_{0}^{T}\int_{\R}\left|\mathcal{V}_{x}\right|\left|\rho(t,x)\right|\dd x\dd t,
\end{eqnarray}
where 
\begin{eqnarray*}
\tilde{\mathcal{S}}_{\mathrm{on}}&=&S_{\mathrm{on}}\left(t,x,q_{\mathrm{on}},\rho,R_{\mathrm{on}}\right)-S_{\mathrm{on}}\left(t,x,\tilde{q}_{\mathrm{on}},\tilde{\rho},\tilde{R}_{\mathrm{on}}\right),\\
\tilde{\mathcal{S}}_{\mathrm{off}}&=&S_{\mathrm{off}}\left(t,x,q_{\mathrm{on}},\rho\right)-S_{\mathrm{off}}\left(t,x,\tilde{q}_{\mathrm{on}},\tilde{\rho}\right),\\
\mathcal{V}&=&v(R)-v(P),\\
\mathcal{V}_{x}&=&\partial_{x}v(R)-\partial_{x}v(P).
\end{eqnarray*}
Except for the second term, all terms appearing at the right-hand side of~\eqref{doubling_variables} are computed as in \cite[Theorem 3.1]{chiarello2022nonlocal},
\begin{eqnarray*}
\int_{0}^{T}\int_{\Omega_{\mathrm{off}}}\left|\mathcal{S}_{\mathrm{off}}\right|\dd x\dd t\nonumber
\label{Soff_bound}&\leq&{\left\|{q}_{\mathrm{off}}\right\|_{\L{\infty}([0,T])}}\int_{0}^{T}\left\|\rho(t,\cdot)-\tilde{\rho}(t,\cdot)\right\|_{\L{1}(\R)}\dd t+{L}\left\|q_{\mathrm{off}}-\tilde{q}_{\mathrm{off}}\right\|_{\L{1}([0,T])},
\end{eqnarray*}
\begin{eqnarray*}
\int_{0}^{T}\int_{\R}\left|\mathcal{V}\right|\left|\partial_{x}\rho(t,x)\right|\dd x \dd t\nonumber
&\leq&\label{diff_vel2}\omega_{\eta}(0)\left\|v'\right\|_{\L{\infty}([0,1])}\sup_{t\in[0,T]} TV(\rho(t,\cdot))\int_{0}^{T}\left\|\rho(t,\cdot)-\tilde{\rho}(t,\cdot)\right\|_{\L{1}(\R)}\dd t,
\end{eqnarray*}
\begin{eqnarray*}
\int_{0}^{T}\int_{\R}\left|\mathcal{V}_{x} \right|\left|\rho(t,x)\right|\dd x\dd t
\label{diff_Vx2}&\leq&\mathcal{W}\int_{0}^{T}\left\|\rho(t,\cdot)-\tilde{\rho}(t,\cdot)\right\|_{\L{1}(\R)}\dd t.
\end{eqnarray*}
Regarding the second term in~\eqref{doubling_variables}, we have  
\begin{eqnarray*}
\int_{0}^{T}\int_{\Omega_{\mathrm{on}}}\left|\tilde{\mathcal{S}}_{\mathrm{on}}\right|\dd x\dd t&=&\int_{0}^{T}\int_{\Omega_{\mathrm{on}}}\left|S_{\mathrm{on}}\left(t,x,q_{\mathrm{on}},\rho,R_{\mathrm{on}}\right)-S_{\mathrm{on}}\left(t,x,\tilde{q}_{\mathrm{on}},\tilde{\rho},\tilde{R}_{\mathrm{on}}\right)\right|\dd x\dd t\\
&\leq&\int_{0}^{T}\int_{\Omega_{\mathrm{on}}}\left(\left|\tilde{\mathcal{S}}_{\mathrm{on}}^{1}\right|+\left|\tilde{\mathcal{S}}_{\mathrm{on}}^{2}\right|+\left|\tilde{\mathcal{S}}_{\mathrm{on}}^{3}\right|\right)\dd x\dd t,
\end{eqnarray*}
where 
\begin{eqnarray*}
\tilde{\mathcal{S}}_{\mathrm{on}}^{1}&=&S_{\mathrm{on}}\left(t,x,q_{\mathrm{on}},\rho,R_{\mathrm{on}}\right)-S_{\mathrm{on}}\left(t,x,q_{\mathrm{on}},\rho,\tilde{R}_{\mathrm{on}}\right),\\
\tilde{\mathcal{S}}_{\mathrm{on}}^{2}&=&S_{\mathrm{on}}\left(t,x,q_{\mathrm{on}},\rho,\tilde{R}_{\mathrm{on}}\right)-S_{\mathrm{on}}\left(t,x,q_{\mathrm{on}},\tilde{\rho},\tilde{R}_{\mathrm{on}}\right),\\ \tilde{\mathcal{S}}_{\mathrm{on}}^{3}&=&S_{\mathrm{on}}\left(t,x,q_{\mathrm{on}},\tilde{\rho},\tilde{R}_{\mathrm{on}}\right)-S_{\mathrm{on}}\left(t,x,\tilde{q}_{\mathrm{on}},\tilde{\rho},\tilde{R}_{\mathrm{on}}\right).
\end{eqnarray*}
We bound $\tilde{\mathcal{S}}_{\mathrm{on}}^{2}$ and $\tilde{\mathcal{S}}_{\mathrm{on}}^{3}$ writing  
\begin{eqnarray*}
\int_{0}^{T}\int_{\Omega_{\mathrm{on}}}\left|\tilde{\mathcal{S}}_{\mathrm{on}}^{2}\right|\dd x\dd t
&\leq& {\left\|q_{\mathrm{on}}\right\|_{\L{\infty}([0,T])}}\int_{0}^{T}\left\|\rho(t,\cdot)-\tilde{\rho}(t,\cdot)\right\|_{\L{1}(\R)}\dd t,
\end{eqnarray*}
\begin{eqnarray*}
\int_{0}^{T}\int_{\Omega_{\mathrm{on}}}\left|\tilde{\mathcal{S}}_{\mathrm{on}}^{3}\right|\dd x\dd t&\leq&\int_{0}^{T}\int_{\Omega_{\mathrm{on}}}\left|q_{\mathrm{on}}-\tilde{q}_{\mathrm{on}}\right| \dd x\dd t\\
&\leq& {L}\left\|q_{\mathrm{on}}-\tilde{q}_{\mathrm{on}}\right\|_{\L{1}([0,T])}.
\end{eqnarray*}
Next, we are going to bound $\tilde{\mathcal{S}}_{\mathrm{on}}^{1}$ term,
\begin{eqnarray*}
\left|\tilde{\mathcal{S}}_{\mathrm{on}}^{1}\right|&=&\left|\mathbf{1}_{\mathrm{on}}q_{\mathrm{on}}\left(1-\rho\right)\left(\left(1-R_{\mathrm{on}}\right)-\left(1-\tilde{R}_{\mathrm{on}}\right)\right)\right|\\
&\leq& {\left\|q_{\mathrm{on}}\right\|_{\L{\infty}([0,T])}}\left|\tilde{R}_{\mathrm{on}}-R_{\mathrm{on}}\right|,
\end{eqnarray*}
thus
\begin{eqnarray*}
\int_{0}^{T}\int_{\Omega_{\mathrm{on}}}\left|\tilde{\mathcal{S}}_{\mathrm{on}}^{1}\right|\dd x\dd t&\leq& {\left\|q_{\mathrm{on}}\right\|_{\L{\infty}([0,T])}}\int_{0}^{T}\int_{\Omega_{\mathrm{on}}}\left|\tilde{R}_{\mathrm{on}}-R_{\mathrm{on}}\right|\dd x\dd t,
\end{eqnarray*}
and observe that 
\begin{eqnarray*}
\int_{\Omega_{\mathrm{on}}}\left|{R}_{\mathrm{on}}-\tilde{R}_{\mathrm{on}}\right|\dd x
&=&\int_{\Omega_{\mathrm{on}}}\left|\omega_{\eta,\delta}\ast\rho-\tilde{\omega}_{\eta,\delta}\ast\tilde{\rho}\right|\dd x\\
&=&\int_{\Omega_{\mathrm{on}}}\left|\omega_{\eta,\delta}\ast\rho-\omega_{\eta,\delta}\ast\tilde{\rho}+\omega_{\eta,\delta}\ast\tilde{\rho}-\tilde{\omega}_{\eta,\delta}\ast\tilde{\rho}\right|\dd x\\
&\leq&\int_{\Omega_{\mathrm{on}}}\left|\omega_{\eta,\delta}\ast\rho-\omega_{\eta,\delta}\ast\tilde{\rho}\right|+\left|\omega_{\eta,\delta}\ast\tilde{\rho}-\tilde{\omega}_{\eta,\delta}\ast\tilde{\rho}\right|\dd x\\
&=&\int_{\Omega_{\mathrm{on}}}\left|\omega_{\eta,\delta}\ast\left(\rho-\tilde{\rho}\right)\right|+\int_{\Omega_{\mathrm{on}}}\left|\left(\omega_{\eta,\delta}-\tilde{\omega}_{\eta,\delta}\right)\ast\tilde{\rho}\right|\dd x\\
&\leq&\left\|\omega_{\eta,\delta}\right\|_{\L{1}(\Omega_{\mathrm{on}})}\left\|\rho-\tilde{\rho}\right\|_{\L{1}(\Omega_{\mathrm{on}})}+\left\|\omega_{\eta,\delta}-\tilde{\omega}_{\eta,\delta}\right\|_{\L{1}(\Omega_{\mathrm{on}})}\left\|\tilde{\rho}\right\|_{\L{1}(\Omega_{\mathrm{on}})}\\
&\leq&\left\|\omega_{\eta,\delta}\right\|_{\L{1}(\R)}\left\|\rho-\tilde{\rho}\right\|_{\L{1}(\R)}+\left\|\omega_{\eta,\delta}-\tilde{\omega}_{\eta,\delta}\right\|_{\L{1}(\R)}\left\|\tilde{\rho}\right\|_{\L{1}(\R)}\\
&\leq&\left\|\rho-\tilde{\rho}\right\|_{\L{1}(\R)}+\left\|\omega_{\eta,\delta}-\tilde{\omega}_{\eta,\delta}\right\|_{\L{1}(\R)}\mathcal{R}_{1}(t),
\end{eqnarray*}
since $\int_{\R}\omega_{\eta,\delta}(x)\dd x=1.$ Here $\mathcal{R}_{1}$ is defined as in~\eqref{R_1}. Then,
\begin{eqnarray*}
\int_{0}^{T}\int_{\Omega_{\mathrm{on}}}\left|\tilde{\mathcal{S}}_{\mathrm{on}}^{1}\right|\dd x\dd t
&\leq& {\left\|q_{\mathrm{on}}\right\|_{\L{\infty}([0,T])}}\int_{0}^{T}\left\|\rho(t,\cdot)-\tilde{\rho}(t,\cdot)\right\|_{\L{1}(\R)}\dd t\\
&&+\int_{0}^{T}\left\|\omega_{\eta,\delta}-\tilde{\omega}_{\eta,\delta}\right\|_{\L{1}(\R)}\mathcal{R}_{1}(t)\dd t\\
&=& \left\|q_{\mathrm{on}}\right\|_{\L{\infty}([0,T])}\int_{0}^{T}\left\|\rho(t,\cdot)-\tilde{\rho}(t,\cdot)\right\|_{\L{1}(\R)}\dd t\\
&&+\left\|\omega_{\eta,\delta}-\tilde{\omega}_{\eta,\delta}\right\|_{\L{1}(\R)}\int_{0}^{T}\mathcal{R}_{1}(t)\dd t.
\end{eqnarray*}
Therefore, the inequality~\eqref{doubling_variables} is equivalent to following inequality
\begin{eqnarray*}
\left\|\rho(T,\cdot)-\tilde{\rho}(T,\cdot)\right\|_{\L{1}(\R)}
&\leq&\left\|\rho_{0}-\tilde{\rho}_{0}\right\|_{\L{1}(\R)}+{L}\left(\left\|q_{\mathrm{on}}-\tilde{q}_{\mathrm{on}}\right\|_{\L{1}([0,t])}+\left\|q_{\mathrm{off}}-\tilde{q}_{\mathrm{off}}\right\|_{\L{1}([0,t])}\right)\nonumber\\
&&+\mathcal{C}\int_{0}^{T}\left\|\rho(t,\cdot)-\tilde{\rho}(t,\cdot)\right\|_{\L{1}(\R)}\dd t+\left\|\omega_{\eta,\delta}-\tilde{\omega}_{\eta,\delta}\right\|_{\L{1}(\R)}\int_{0}^{T}\mathcal{R}_{1}(t)\dd t\\
&=&\left\|\rho_{0}-\tilde{\rho}_{0}\right\|_{\L{1}(\R)}+{L}\left(\left\|q_{\mathrm{on}}-\tilde{q}_{\mathrm{on}}\right\|_{\L{1}([0,t])}+\left\|q_{\mathrm{off}}-\tilde{q}_{\mathrm{off}}\right\|_{\L{1}([0,t])}\right)\nonumber\\
&&+\mathcal{C}\int_{0}^{T}\left\|\rho(t,\cdot)-\tilde{\rho}(t,\cdot)\right\|_{\L{1}(\R)}\dd t+r(T)\left\|\omega_{\eta,\delta}-\tilde{\omega}_{\eta,\delta}\right\|_{\L{1}(\R)}.
\end{eqnarray*}
An application of Gronwall's Lemma to the above inequality completes the proof.
\end{proof}

\section{Numerical examples}\label{sec:Num_exp}
In this Section we use \cite[Algorithm 3.1]{chiarello2022nonlocal} in order to compute approximate solutions to \eqref{nonRTM_2}-\eqref{initial_condition} with the source term \eqref{Son_Hyp}. We simulate an optimization problem in traffic merging, which consists in finding the optimal value $\delta$ keeping fixed $\eta$ in the kernel function $\omega_{\eta,\delta}.$  According to the meaning of that kernel function, this is equivalent to find where a drivers located on the on-ramp should look in order to avoid the congestion on the main road. To this end, we consider one on-ramp with length $L=0.1$ located from $x=1.0$ until $x=1.1$, and we  consider the following kernel functions 
\begin{eqnarray*}
\omega_{\eta}(x)&:=&2\frac{\eta-x}{\eta^2}\chi_{[0,\eta]}(x),\\
\omega_{\eta,\delta}(x)&:=& \frac{1}{\eta^6}\frac{16}{5\pi}\left(\eta^2-(x-\delta)^2 \right)^{5/2}\chi_{[-\eta+\delta,\eta+\delta]}(x),
\end{eqnarray*}
for convective and reactive terms, respectively, with $\eta=0.5$ and $\delta\in[-\eta,\eta]$, for $x\in[-1,4]$ at simulated time $T=6$ and velocity function given by $v(\rho)=1-\rho$. We also consider $\Delta x=1/200$, a constant initial condition $\rho_{0}(x)=0.3$, and the rate of the on-ramp given by $q_{\mathrm{on}}(t)=1.2$. Following \cite{chiarello2019stability}, as a metric of traffic congestion we consider the two following functionals
\begin{eqnarray*}
    J(T)=\int_{0}^{T}\mathrm{d}|\partial_{x}\rho(t,\cdot)|\mathrm{d}t, 
\end{eqnarray*}
\begin{eqnarray*}
\Psi(T;a,b)=\int_{0}^{T}\int_{a}^{b}\varphi(\rho(t,x))\mathrm{d}x\mathrm{d}t,
\end{eqnarray*}
where
\begin{eqnarray*}
\varphi(r)= \left\{ \begin{array}{lcc}
             0, &    r<0.75, \\
             10r-7.5,& 0.75\leq r\leq0.85, \\
             1, & 0.85<r\leq1.
             \end{array}
   \right.
   \end{eqnarray*}
The functional J defined measures the integral with respect to time of the spatial total variation of the traffic density while
the functional $\Psi$ measures the traffic congestion in the interval $[a,b]=[-1,4]$.
Figure~\ref{fig:psi_vs_delta} shows the values of the functionals $J$ and $\Psi$ when we vary the value of $\delta,$ keeping fixed $\eta$. We can observe that the minimum value of $\Psi$ is given when $\delta=0.1$. In Figure~\ref{fig:x_vs_rho} we compare the solutions of~\eqref{nonRTM_2}-\eqref{initial_condition} for different values of $\delta$, keeping fixed $\eta.$ Namely, we consider $\delta\in\{-0.5,0.1,0.5\}$ and $\eta=0.5$. We can see that the solution for $\delta=0.1$ 
produces a smaller queue and therefore less increase of the density on the main road when vehicles enter through the on-ramp than the other values considered for $\delta$. 
\begin{figure}[htbp]
     \centering
     \includegraphics[scale=0.5]{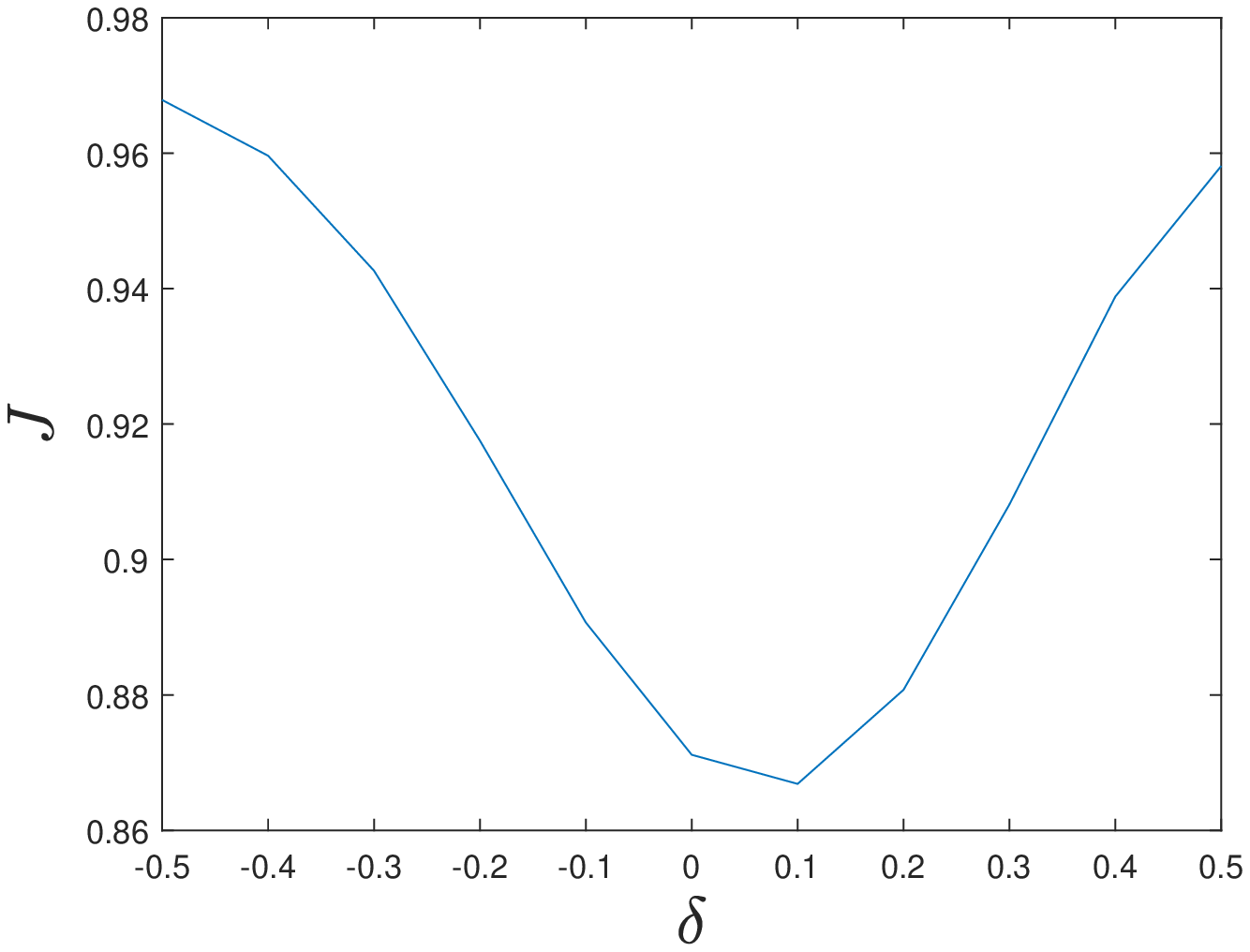}\includegraphics[scale=0.5]{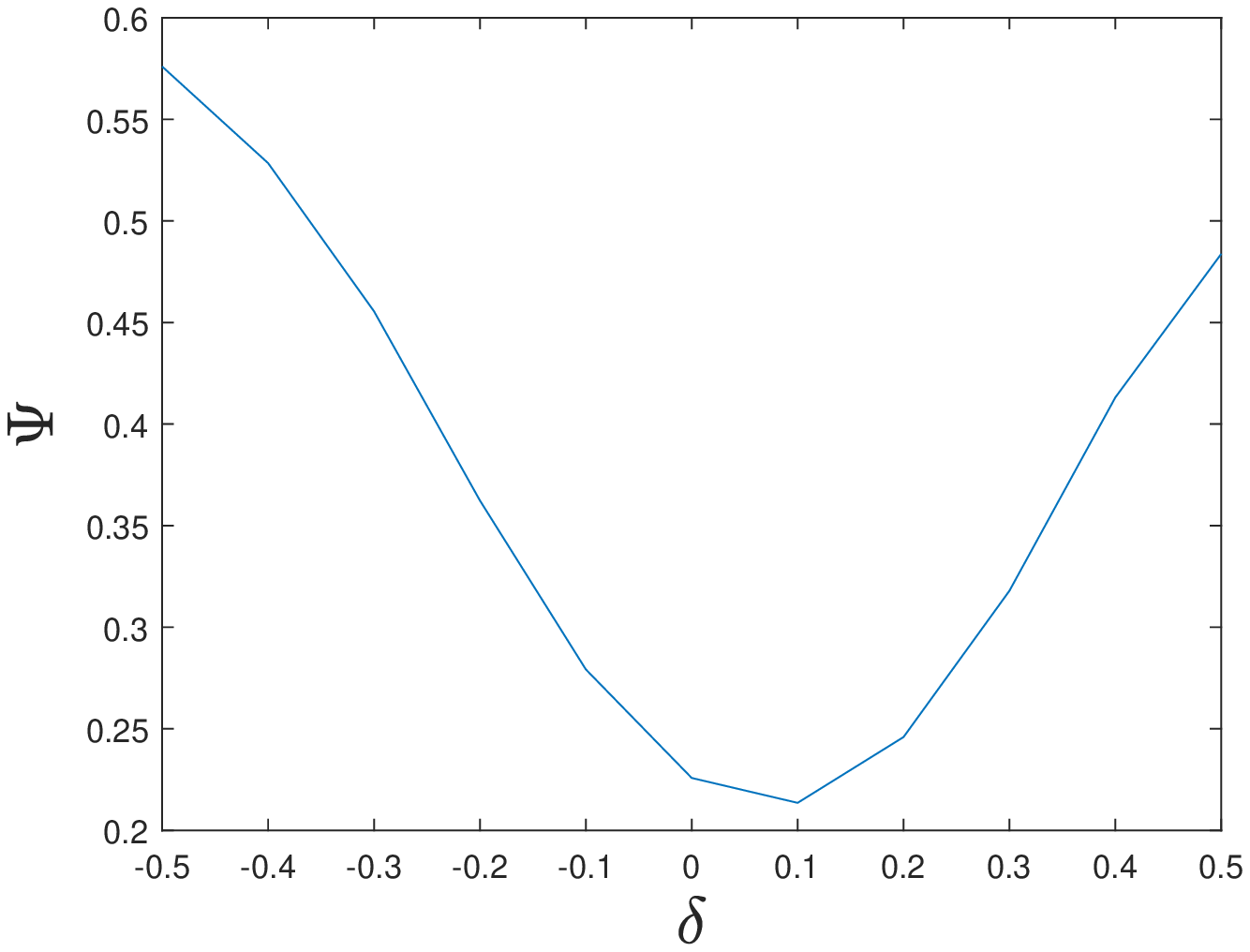}  
     \caption{Left: Functional  $J$ with $\eta=0.5$ and $\delta\in[-\eta,\eta]$.  Right: Functional $\Psi$ with $\eta=0.5$ and $\delta\in[-\eta,\eta]$}
     \label{fig:psi_vs_delta}
 \end{figure}

\begin{figure}[htbp]
     \centering
     \includegraphics[scale=0.5]{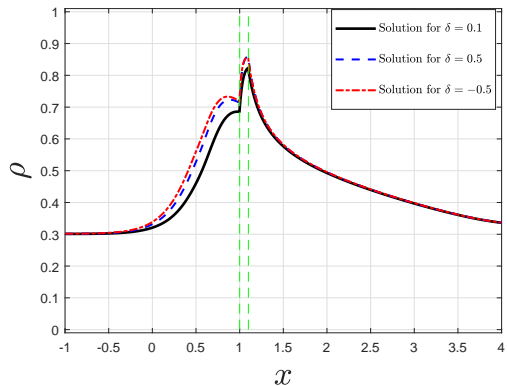}
     \caption{Solution of \eqref{nonRTM_2}-\eqref{initial_condition} for $\eta=0.5$ and varying $\delta\in\{-0.5,0.1,0.5\}$}
     \label{fig:x_vs_rho}
 \end{figure}

\section{Conclusions}
In this work we proved the $\L1-$stability of entropy weak solutions of a scalar nonlocal balance law with nonlocal source term arising in traffic modelling with on- and off-ramps introduced in~\cite{chiarello2022nonlocal}. We reached an estimate of the dependence of the solution with respect to the kernel function in the source term, the on-ramp rate, the off-ramp rate and the initial datum. The stability has been obtained from the entropy condition through the doubling of variables technique. Finally, following~\cite{chiarello2019stability}, we have shown numerical simulations illustrating the dependencies above for two cost functionals derived from traffic flow applications. 
\section*{Acknowledgments}
FAC is member of the Gruppo Nazionale per l'Analisi Matematica, la Probabilit\`a e le loro Applicazioni (GNAMPA) of the Istituto Nazionale di Alta Matematica (INdAM).  HDC are supported by the INRIA Associated Team ``Efficient numerical schemes for non-local transport phenomena'' (NOLOCO; 2018--2020). HDC was partially supported by the National Agency for Research and Development, ANID-Chile through Scholarship Program, Doctorado Becas Chile 2021, 21210826 and by Anillo project ACT210030. The authors would like to thank Luis Miguel Villada for helpful discussions about this work. 
{
\bibliographystyle{siam}

}

\end{document}